\documentclass{article}
\usepackage{amsrefs,amssymb,amsmath,amsthm,amsfonts,epsfig,graphicx}
\usepackage[utf8]{inputenc}
\usepackage[T1]{fontenc}
\usepackage[english]{babel}
\usepackage{hyperref}
\hypersetup{
    colorlinks,
    citecolor=black,
    filecolor=black,
    linkcolor=black,
    urlcolor=black
}

\theoremstyle{plain}
\newtheorem{teo}{Theorem}[section]

\newtheorem{prop}[teo]{Proposition}		
\newtheorem{rmk}[teo]{Remark}

\theoremstyle{definition}
\newtheorem{df}[teo]{{Definition}}

\DeclareMathOperator{\diam}{diam}

\DeclareMathOperator{\dist}{dist}

\newcommand{\expc}{\eta}

\newcommand{\R}    {\mathbb R}

\newcommand{\Z}  {\mathbb Z}

\renewcommand{\epsilon}{\varepsilon}

\title{Minimal expansive systems and spiral points}
\author{Alfonso Artigue}

\begin{document}

\maketitle

\begin{abstract}
It is known that if a compact metric space $X$ admits a minimal expansive homeomorphism 
then $X$ is totally disconnected. In this note we give a short proof of this result 
and we analyze its extension to expansive flows.
\end{abstract}

\section{Introduction}
In \cite{Ma} Mañé proved that if a compact metric space $(X,\dist)$ admits a minimal expansive homeomorphism 
then $X$ is \emph{totally disconnected}, i.e., every connected subset of $X$ is a singleton.
A homeomorphism $f\colon X\to X$ is \emph{expansive} if there is 
$\expc>0$ such that if $\dist(f^n(x),f^n(y))<\expc$ for all $n\in\Z$ then $y=x$. 
We say that $f$ is \emph{minimal} if the set $\{f^n(x):n\geq 0\}$ is dense in $X$ for all $x\in X$.
In Theorem \ref{teoMaKaMin} we give a short proof of Mañé's result.

Some definitions, results and proofs on discrete-time dynamical systems 
are trivially extended to flows, i.e., real actions $\phi\colon\R\times X\to X$. 
But, this is not the case for the problem considered in this note. 
In the context of flows we say that $\phi$ is \emph{expansive} \cite{BW} if for all 
$\epsilon>0$ there is $\expc>0$ such that if $\dist(\phi_t(x),\phi_{h(t)}(y))<\expc$ 
for some increasing homeomorphism $h\colon\R\to\R$, $h(0)=0$, 
then $y=\phi_s(x)$ for some $s\in (-\epsilon,\epsilon)$. 
In the case of expansive flows, Mañé's proof was partially extended by Keynes and Sears \cite{KS}.
They proved that if $X$ admits a minimal expansive flow \emph{without spiral points} (see Section \ref{secExpFlow}) 
then the topological dimension of $X$ 
is 1, i.e., local cross sections are totally disconnected.

Our proof of Theorem \ref{teoMaKaMin} avoids the problem of spiral points in the discrete-time case. 
However, its extension to flows is far from trivial, at least for the author. 
In the final section some remarks are given in relation with spiral points for homeomorphisms and flows.

\section{Minimal homeomorphisms}
Let $(X,\dist)$ be a compact metric space and assume that 
$f\colon X\to X$ is a homeomorphism. 
Let us start proving a general {result not} involving expansivity. 
Recall that a \emph{continuum} is a compact connected set.
A continuum is non-trivial if it is not a singleton.
We denote by $\dim(X)$ the topological dimension of $X$ \cite{HW}. 
For the next result we need to know that $\dim(X)>0$ if and only if $X$ has a non-trivial connected component.

\begin{rmk}
\label{rmkSubCont}
  {For a continuum $C$ it is known that the space of its subcontinua with the Hausdorff metric is arc connected, see \cite{IN}. 
Therefore, since the diameter function is continuous, given $0<\delta<\diam(C)$ there is a subcontinuum $C'\subset C$ 
with $\diam(C')=\delta$.}
\end{rmk}

\begin{prop}
\label{propContEst}
If $\dim(X)>0$ then for all $\delta>0$ there is a non-trivial continuum $C\subset X$ such that 
$\diam(f^n(C))\leq\delta$ for all $n\geq 0$ or for all $n\leq 0$.
\end{prop}

\begin{proof}
By contradiction assume that there are $\delta>0$, a sequence of continua $C_n$ and 
positive integers $k_n$ such that $\diam(C_n)\to 0$ and $\diam(f^{k_n}(C_n))>\delta$. 
Considering subcontinua and different times $k_n$, we can also assume that 
$\diam(f^{k_n}(C_n))=\delta$ {(Remark \ref{rmkSubCont})} and $\diam(f^i(C_n))\leq\delta$ if $0\leq i\leq k_n$. 
Since $\diam(C_n)\to 0$ we have that $k_n\to+\infty$. 
Take a limit continuum $C$ of {$f^{k_n}(C_n)$} in the Hausdorff metric. 
By the continuity of $f$ we have that $\diam(f^i(C))\leq\delta$ for all $i\leq 0$. 
This proves the result.
\end{proof}

\begin{rmk}
  It can be the case that for every non-trivial continuum {$C\neq X$} it holds that: 
  $\lim_{n\to+\infty}\diam(f^n(C))=\diam(X)$ and $\lim_{n\to-\infty}\diam(f^n(C))=0$. 
  This is the case on an expanding attractor as for example the solenoid and the nonwandering set 
  of a derived from Anosov.
\end{rmk}

In \cite{Ka93} Kato introduced a generalization of expansivity that allowed him to extend several 
results of expansive homeomorphisms, including those obtained in \cite{Ma}. 
We recall that $f$ is \emph{cw-expansive} {(\emph{continuum-wise expansive})}, 
if there is $\expc>0$ such that if $C\subset X$ is connected and $\diam(f^n(C))<\expc$ for all $n\in\Z$ 
then $C$ is a singleton.
In this case we say that $\expc$ is a \emph{cw-expansive constant} for $f$.
We say that $C\subset X$ is $\expc$-\emph{stable} if $\diam(f^i(C))\leq\expc$ for all $i\geq 0$.

\begin{teo}{\cites{Ma,Ka93}}
\label{teoMaKaMin}
If $f\colon X\to X$ is a minimal homeomorphism of a compact metric space $(X,\dist)$ and $\dim(X)>0$ then for all 
$\expc>0$ there is a non-trivial continuum $C\subset X$ such that $\diam(f^n(C))\leq\expc$ for all $n\in\Z$.
\end{teo}

\begin{proof}
Arguing by contradiction assume that $f$ is minimal, cw-expansive and $\dim(X)>0$. 
Consider a cw-expansive constant $\expc>0$. 
It is known \cite{Ka93} that in this case there is $m>0$ such that if $C$ 
is a $\expc$-stable continuum and $\diam(C)=\expc/3$ then
$\diam(f^{-m}(C))>\expc$. The proof is similar to the one of Proposition \ref{propContEst}.

Let $Y\subset X$ be a minimal set for $g=f^m$. 
In this paragraph we will show that $\dim(Y)=0$.
Take an open set $U\subset Y$ such that $\diam(U)<\expc/3$. 
By contradiction assume that $\dim(Y)>0$. 
By Proposition \ref{propContEst} there is a $\expc$-stable 
continuum $C_0\subset Y$ for $g$ (the case of $g^{-1}$ is analogous). 
Taking a negative iterate of $C_0$ we can assume that $\diam(C_0)\geq\expc/3$.
Since $\diam(g^{-1}(C_0))>\expc$ 
and $\diam(U)<\expc/3$ there is a component $C_1$ of $g^{-1}(C_0)\setminus U$ with 
$\diam(C_1)\geq \expc/3$. 
In this way we construct a sequence of continua $C_k$ such that $\diam(C_k)\geq\expc/3$, 
$C_{k+1}\subset g^{-1}(C_k)$ and $C_k\cap U=\emptyset$ for all $k\geq 0$. 
Take $x\in \cap_{k\geq 0}g^{k}(C_k)$. 
Then $g^{-k}(x)\notin U$ for all $k\geq 0$.
This contradicts that $Y$ is minimal for $g$ and proves that $\dim(Y)=0$.

Now, since $f\colon X\to X$ is minimal, we have that $X=\cup_{i=1}^mf^i(Y)$ (disjoint union). 
Since $\dim(Y)=0$ we conclude that $\dim(X)=0$ and the proof ends.
\end{proof}

\begin{rmk}
Let us illustrate the possible behavior of the diameter of the iterates of a continuum $C$ as in the previous theorem. 
If $f$ is an irrational rotation of a circle then $\diam(f^n(C))$ is constant (for a suitable metric) for all $n\in\Z$. 
In \cite{BKS} a minimal set on the two-dimensional torus is given containing a circle $C$ 
such that $\diam(f^n(C))\to 0$ as $n\to\pm\infty$.
\end{rmk}

\begin{rmk}
In \cite{F} Floyd gave an example of a compact subset $X\subset \R^2$ and a minimal homeomorphism $f\colon X\to X$. 
Each connected component of $X$ is an {interval, some} of them trivial, i.e., {singletons}. 
Therefore, $X$ is a nonhomogeneous space of positive topological dimension. 
This example also shows that on a minimal set there can be points not belonging to a non-trivial stable continuum 
(the trivial components). 
Assuming that $X$ is a Peano continuum and $f\colon X\to X$ is a minimal homeomorphism, 
is it true that every point belongs to a stable continuum?
\end{rmk}

We do not know if the proof of Theorem \ref{teoMaKaMin} can be \emph{translated} to expansive flows.
Another way to translate the result is solving the spiral point problem indicated in the next section. 

%
%
%
%
%

\section{Spiral points}
\label{secExpFlow}

In this section we wish to discuss the problem of \emph{translating} Mañé's proof for minimal expansive flows. 
As we said, it is related {to} spiral points. 
We start giving some remarks for spiral points of homeomorphisms, next we consider the corresponding concept for flows.

\begin{df}
  Let $f\colon X\to X$ be a homeomorphism of a metric space $(X,\dist)$. 
  We say that $x\in X$ is a \emph{spiral point} if there is 
  $m\neq 0$ such that $$\lim_{n\to +\infty}\dist(f^n(x),f^{n+m}(x))=0.$$
\end{df}

The proofs by Mañé and Kato of Theorem \ref{teoMaKaMin} use a simple lemma, \cite{Ma}*{Lemma II} and \cite{Ka93}*{Lemma 5.3}.
This lemma {says} that if $x$ is a spiral point 
then $f$ has a periodic point. Therefore, non-trivial minimal sets have no spiral points.
Let us give some details on this result. 

\begin{prop}
  If a homeomorphism of a compact metric space $f\colon X\to X$ has a spiral point $x\in X$ then 
  $\omega(x)$ is a finite union of continua $C_1,\dots,C_m$ such that 
  $f(C_i)=C_{i+1}$ for $i=1,\dots,m-1$, $f(C_m)=C_1$ 
  and $f^m\colon C_i\to C_i$ is the identity.
\end{prop}

\begin{proof}
Let $C_1$ be the $\omega$-limit set of $x$ by $f^m$. 
Since $x$ is a spiral point the continuity of $f$ implies that $f^m(y)=y$ for all $y\in C_1$. 
The sets $C_i$ have to be defined as $C_i=f^{i-1}(C_1)$. 
It only {remains} to prove that $C_1$ is connected. 
By contradiction assume that $C_1$ is a disjoint union of two non-empty 
compact sets $D,E\subset X$. 
Define $r=\inf\{\dist(p,q):p\in D, q\in E\}>0$ and $g=f^m$.
Take $n_0$ such that for all $n\geq n_0$ it holds that $g^n(x)\in B_{r/3}(D)\cup B_{r/3}(E)$ 
and $\dist(g^n(x),g^{n+1}(x))<r/3$. 
If $g^{n_0}(x)\in B_{r/3}(D)$ we have that $g^{n}(x)\in B_{r/3}(D)$ for all $n\geq n_0$. 
In this case $E$ is empty. If $g^{n_0}(x)\in B_{r/3}(E)$ we conclude that $D$ is empty. 
This gives a contradiction that proves the result.
\end{proof}

\begin{rmk}
\label{rmkNontrivialSpiral}
Trivial examples of spiral points are points in the stable set of a hyperbolic periodic point. 
  Let us give an example showing that the $\omega$-limit set of a spiral point may not be a finite set. 
  Take a sequence $a_n\in \R$ such that $a_n\to+\infty$ and  $|a_{n+1}-a_n|\to 0$ as $n\to+\infty$
  (for example $a_n=\sqrt{n}$). 
  Define a sequence of points in $\R^2$ by $x_n=(1/n,\sin(a_n))$ for $n\geq 1$. 
  It is easy to prove that $\dist(x_{n+1},x_n)\to 0$ as $n\to+\infty$ (Euclidean metric in $\R^2$). 
  Consider $X\subset\R^2$ a compact set containing $\{x_n:n\geq 1\}$ such that 
  there is $f\colon X\to X$ satisfying $f(x_n)=x_{n+1}$ for all $n\geq 1$. 
  By construction we have that the $\omega$-limit set by $f$ of $x_1\in X$ is 
  $$\omega(x_1)=\{(0,y)\in \R^2:|y|\leq 1\}.$$
  We have that $x_1$ is a spiral point and its $\omega$-limit set is formed by fixed points of $f$. 
\end{rmk}

We now consider the corresponding problem for flows.
Consider $\phi\colon \R\times X\to X$ a continuous flow of the metric space $(X,\dist)$.

\begin{df}
  We say that $x\in X$ is a \emph{kinematic spiral point} if there is 
  $\tau\neq 0$ such that $\dist(\phi_t(x),\phi_{t+\tau}(x))\to 0$ as $t\to+\infty$.
\end{df}

\begin{prop}
  If a continuous flow of a compact metric space $\phi\colon \R\times X\to X$ 
  has a kinematic spiral point $x\in X$ then $\phi_\tau(p)=p$ for all $p\in\omega(x)$.
\end{prop}

\begin{proof}
  It is direct from the definitions.
\end{proof}

This implies that a non-trivial minimal set cannot have kinematic spiral points.
We now remark that Theorem \ref{teoMaKaMin} cannot be extended for kinematic expansive flows.

\begin{rmk}
A flow $\phi$ on a compact metric space $X$ 
is \emph{kinematic expansive} if for all $\epsilon>0$ there is $\expc>0$ such that 
if $\dist(\phi_t(x),{\phi}_t(y))<\expc$ for all $t\in\R$ then $y=\phi_s(x)$ for some $s\in(-\epsilon,\epsilon)$.
This definition is weaker than Bowen-Walters expansivity. 
Some non-trivial minimal examples are known. 
In \cite{Mat} Matsumoto proves that the two-torus admits 
a $C^0$ minimal kinematic expansive flow. 
In \cite{ArK} it is shown that no $C^1$ minimal 
flow on the two-torus is kinematic expansive.
On three-dimensional manifolds it is known that the horocycle flow of a compact surface of constant negative 
curvature is minimal. In \cite{Gura} Gura shows that this flow is \emph{separating}, i.e., 
there is $\expc>0$ such that if $\dist(\phi_t(x),\phi_t(y))<\expc$ for all $t\in\R$ then $y=\phi_s(x)$ for some $s\in\R$.
It {seems that} the horocycle flow is in fact kinematic expansive.
\end{rmk}

Since the definition of expansive flow introduced by Bowen and Walters is 
stated using time reparameterizations the following definition is natural.

\begin{df}
We say that $x\in X$ is a \emph{spiral point}
if there is a continuous function $h\colon \R\to\R$  
and $\tau>0$ such that 
$h(t)-t >\tau$ for all $t\geq 0$ and $\dist(\phi_t(x),\phi_{h(t)}(x))\to 0$ as 
$t\to+\infty$. 
\end{df}

\begin{rmk}
  In \cite{KS} the definition of spiral point is introduced using local cross sections 
  of the flow but the definitions are equivalent.
\end{rmk}

\begin{rmk}
  If there is $T>\tau$ such that $T>h(t)-t>\tau$ for all $t\geq 0$ (in the previous definition) 
  then every orbit in the $\omega$-limit set of $x$ is {periodic or singular. 
  This $\omega$-limit set may not be a single orbit. 
  Consider, for example, the suspension of the homeomorphism in Remark \ref{rmkNontrivialSpiral}.}
\end{rmk}

Spiral points appear in the Poincaré-Bendixon theory of two-dimensional flows.
In fact, if $\phi$ is a continuous flow on the two-dimensional sphere then every point is spiral.
{A \emph{closed orbit} is a closed subset of $X$ of the form $\{\phi_t(x):t\in\R\}$. 
Since $X$ is compact, every closed orbit is a periodic orbit or a stationary point.}
\begin{rmk}
It can be the case that the $\omega$-limit set of a spiral point contains non-closed orbits. 
An example is illustrated in Figure \ref{figSpiral}. 
In this case $h(t)-t\to+\infty$ as $t\to+\infty$. 
However, there is a fixed point (a special case of closed orbit) in the $\omega$-limit set of this spiral point.
\begin{figure}[h]
\center
\includegraphics[scale=.7]{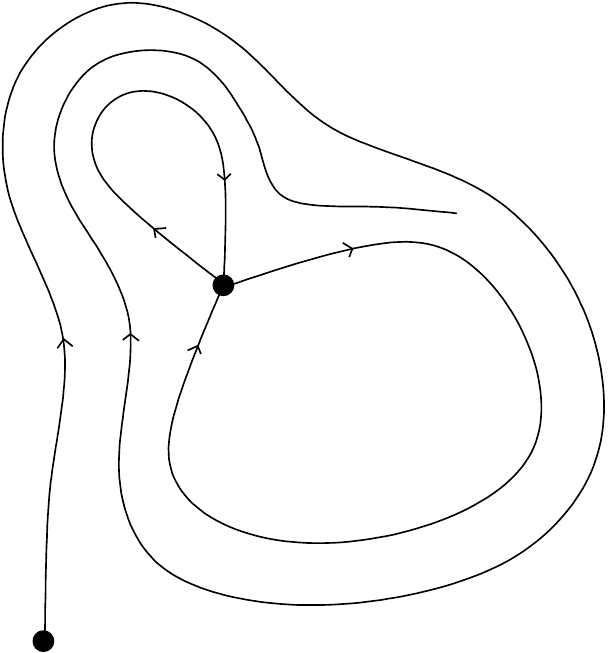}  
\caption{A spiral point in the plane.}
\label{figSpiral}
\end{figure}
\end{rmk}

Let us state the following question: if $\phi$ is a flow on a compact metric space $X$ and 
$x\in X$ is a spiral point, is there a closed orbit in $\omega(x)$?
A positive answer, even assuming that $\phi$ is a minimal expansive flow, would prove (using \cite{KS}) 
that if a compact metric {space} admits a minimal expansive flow then $\dim(X)\leq 1$.

\noindent Departamento de Matemática y Estadística del Litoral, Salto-Uruguay\\
Universidad de la República\\
E-mail: artigue@unorte.edu.uy
\end{document}